\documentclass[preprint,12pt,authoryear]{elsarticle}

\usepackage{amssymb}
\usepackage{amsmath}
\usepackage{amsthm}

\usepackage[utf8]{inputenc}
\usepackage{geometry}
\usepackage{booktabs}  
\usepackage{graphicx}
\usepackage{hyperref}
\usepackage{color}
\usepackage{lineno}    
\usepackage{soul}
\newtheorem{theorem}{Theorem}

\journal{Journal of Fluids and Structures}

\begin{document}

\begin{frontmatter}



\title{Tropical Geometry as a Restricted Architecture for Physics-Informed Neural Networks: Applications in Nonlinear Fluid-Structure Examples}

\author[inst1]{Carla Valencia-Negrete\corref{cor1}}
\ead{carla.valencia@ibero.mx} 
\author[inst2]{Cristhian Garay-Lopez}
\author[inst1]{Marco Favela-Rodriguez}
\author[inst1]{Alonso Andapia-Viveros}

\cortext[cor1]{Corresponding author.}

\affiliation[inst1]{organization={Department of Physics and Mathematics, Universidad Iberoamericana},
             addressline={Prolongacion Paseo de la Reforma 880},
             city={Mexico City},
             postcode={01219},
             country={Mexico}}
 \affiliation[inst2]{organization={Centro de Investigacion en Matematicas, A.C.},
            addressline={Jalisco S/N, Col. Valenciana},
             city={Guanajuato},
             postcode={36023},
             country={Mexico}}

\begin{abstract}
Nonlinear algebraic (polynomial) differential equations that govern fluid-structure interactions, such as those modeling vortex-induced vibrations, and shock waves, often lack analytical solutions, creating significant challenges to efficient prediction and control. While Physics-Informed Neural Networks (PINNs) offer a mesh-free numerical alternative, they frequently suffer from convergence stagnation when optimizing over chaotic landscapes or stiff singularities. This paper introduces a hybrid methodology that integrates tropical differential algebraic geometry with deep learning. Using tropical algebra, we algorithmically determine a hard constraint, which we use to restrict the neural network's hypothesis space to the exact support of the valid formal power series solution. We establish a theoretical Valuation-Support equivalence between classical Briot-Bouquet indicial analysis and the fundamental theorem of tropical differential algebraic geometry, proving that tropical methods accurately identify singularity structures. Numerical experiments on the Van der Pol and Burgers' equations demonstrate that embedding these tropical constraints directly into the network architecture drastically reduces the search space, overcoming optimization stagnation and improving both accuracy and convergence speed in non-homogeneous physical regimes.
\end{abstract}

\begin{graphicalabstract}
\includegraphics[scale=0.45]{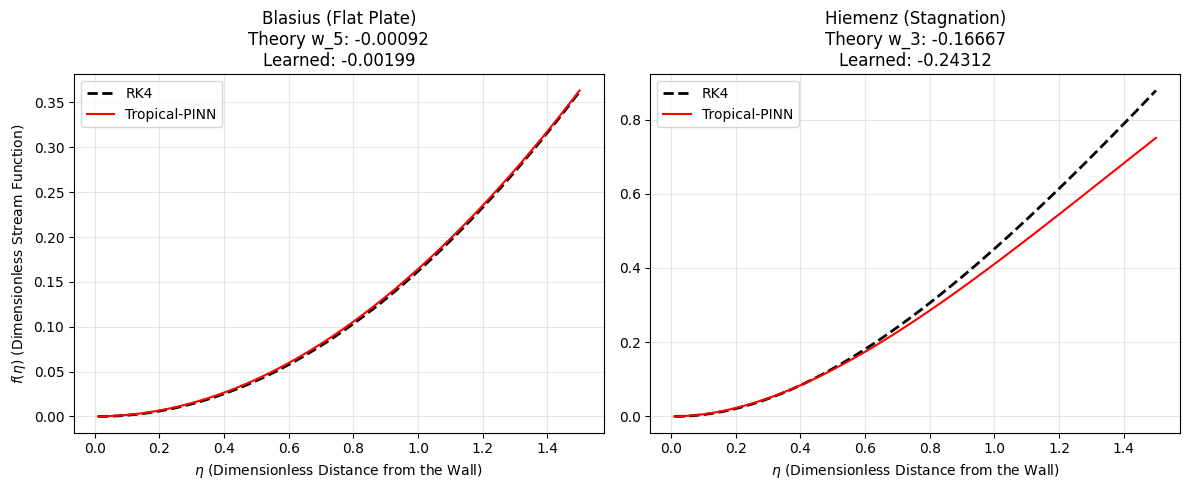}
\end{graphicalabstract}

\begin{highlights}
\item Proves the exact algebraic equivalence between classical Briot-Bouquet indicial analysis and tropical differential algebraic geometry.
\item Hard-coding tropical skeletons in the network architecture bypasses chaotic optimization landscapes and stiff singularities.
\item The hybrid Tropical-PINN model extracts highly precise, physical Taylor series coefficients directly from the physics loss.
\item Successfully identifies structural phase shifts in fluid-structure models, including the Blasius, Hiemenz, and Viscous Burgers' equations.
\end{highlights}


\end{frontmatter}


\section{Introduction}
\label{sec:sample1}
The mathematical framework for capturing complex fluid–structure interactions is part of the theory of nonlinear algebraic (polynomial) differential equations. A key difficulty in analyzing nonlinear models arises at points where the Jacobian matrix, our linear basis for behavioral analysis, has null eigenvalues. Historically, the structural analysis of such singularities has relied on the classical Briot-Bouquet indicial analysis \citep{Iwano1966} and the Painlevé test \citep{Painleve1900,Painleve1902}, which seek formal power series solutions by balancing dominant terms, which is computationally costly.

In these cases, solutions may not exist, and instead of uniqueness, we have multiple possible expressions that satisfy the equation  \citep{Halburd2025}. Since these equations often do not admit closed-form analytical solutions, researchers primarily depend on numerical approximation methods \citep{Hairer1991,Jay2005,Wang2019}. Physics-Informed Neural Networks (PINNs) have gained attention as an alternative to conventional numerical solvers, encoding the governing physical principles directly into the loss function \citep{Raissi2019}. However, when exploring high-dimensional landscapes associated with stiff singularities or topological bifurcations, standard PINNs often do not converge \citep{Wang2021,Mojgani2022}.

This paper introduces a hybrid computational methodology that bridges this gap: the Tropical-PINN. By algorithmically identifying the exact  supports of the space of (formal power series) solutions, the tropical framework restricts the network’s hypothesis space before optimization begins, resolving the spectral bias and optimization stagnation that plague standard PINNs.

In this work, we first establish a formal Valuation-Support equivalence, showing that the classical indicial equations \citep{Conte1993,Wu2012} are the algebraic shadow of tropical evaluation operators. We then demonstrate that embedding these algorithmically generated topological priors into a neural network resolves the coefficient imprecision of tropical methods. Through numerical experiments on the van der Pol \citep{Liu2024,Hardika2024}, Blasius \citep{Bararnia2022}, KdV-Burgers \citep{Baumik2024}, and Falkner-Skan equations \citep{Schlichting2017}, we show that this hybrid approach has structural validity, capturing physical phase shifts, such as the transition from sparse Blasius flow to dense Hiemenz stagnation, and recovering highly precise coefficients across diverse, non-homogeneous regimes.

\section{Theoretical Framework}
\label{sec:theory}
The source for this section is \citep{AGT}. Consider the ring of polynomials $\mathbb{C}[y_i\::\:i\in\mathbb{N}]$, whose elements are of the form $F=a_1E_1+\cdots+a_mE_m$, where for every $k=1,\ldots,m$, we have $a_k\in\mathbb{C}\setminus\{0\}=\mathbb{C}^*$ and $E_k=\prod_{0\leq i\leq r} y_i^{m_{i,k}}$ for some $r\geq0$ and $m_{i,k}\geq0$.

Consider the derivation  $d:\mathbb{C}[y_i\::\:i\in\mathbb{N}]\xrightarrow[]{}\mathbb{C}[y_i\::\:i\in\mathbb{N}]$ defined on the basis by  $d(y_i)=y_{i+1}$ (then extend following linearity and the product rule). The pair $\mathbb{C}\{y\}=(\mathbb{C}[y_i\::\:i\in\mathbb{N}],d)$ is the differential ring of differential polynomials (or algebraic differential equations) with coefficients in $\mathbb{C}$ in one differential variable $y$.

Let $\mathbb{B}=\{0_{\mathbb{B}}<1_{\mathbb{B}}\}$ be the Boolean semiring, satisfying $0_\mathbb{B}x=0_\mathbb{B}$, $1_\mathbb{B}x=x$ for $x\in\mathbb{B}$, and $x+y:=\max\{x,y\}$. We denote by $v_0:\mathbb{C}\xrightarrow[]{}\mathbb{B}$ the trivial valuation, defined by $v_0(0)=0_\mathbb{B}$ and $v(a)=1_\mathbb{B}$ for all $a\neq0.$ The tropicalization (or support) of $F=\sum_{k=1}^ma_iE_i$ is the formal expression $supp(F)=v_0(a_1)E_1+\cdots+v_0(a_m)E_m$.

Let $\mathbb{C}[\![t]\!]$ be the ring of formal power series with coefficients in $\mathbb{C}$. The tropicalization (or support) of $\sum_ia_it^i=\Phi\in \mathbb{C}[\![t]\!]$ is the formal expression $supp(\Phi)=\sum_iv_0(a_i)t^i$. We denote by $ord:\mathbb{C}[\![t]\!]\xrightarrow[]{}\mathbb{N}\cup\{\infty\}$ the $t$-adic valuation, defined by $ord(0)=\infty$, and $ord(\sum_ia_it^i)=\min supp(\Phi)$ otherwise.

Now consider $E=\prod_{0\leq i\leq r} x_i^{m_{i}}$ and $A\subset\mathbb{N}$.  Given the formal expressions $1_\mathbb{B}E(\varphi)$ and $\varphi=\sum_{i\in A}1_\mathbb{B}t^i$, one can define the evaluation $1_\mathbb{B}E(\varphi)$ by: 

\begin{equation}
\label{eq:trop_eval}
    1_\mathbb{B}E(\varphi):=\sum_{i=0}^rm_ival_i(\varphi)\in \mathbb{N}\cup\{\infty\},
\end{equation}

where for $i\geq0$ we define $val_i(\varphi)=min\{a-i\geq0\::\:a\in A\}$, if such minimum exists, and $\infty$ otherwise. If $\mathfrak{P}=1_\mathbb{B}E_1+\cdots+1_\mathbb{B}E_m$ and $\varphi=\sum_{i\in A}1_\mathbb{B}t^i$, we define  $\mathfrak{P}(\varphi):=min_{1\leq k\leq m}\{1_\mathbb{B}E_k(\varphi)\},$
and we say that $\varphi$ is a tropical solution to $\mathfrak{P}$ if there exists $1\leq i\neq j\leq m$ such that $\mathfrak{P}(\varphi)=1_\mathbb{B}E_i(\varphi)=1_\mathbb{B}E_j(\varphi)$. We denote by $Sol(\mathfrak{P})$ the set of all the formal expressions $\varphi=\sum_{i\in A}1_\mathbb{B}t^i$ which are solutions to $\mathfrak{P}$.

If $F=\sum_{k=1}^ma_iE_i$ is an algebraic differential equation, we denote by $Sol(F)$ the set of formal power series solutions of the system $\{F=0\}$, this is $Sol(F)=\{\varphi\in \mathbb{C}[\![t]\!]\::\:F(\varphi)=0\}$. The Fundamental Theorem of Tropical Differential Algebraic Geometry says that if $G\subset \mathbb{C}\{y\}$ is a differential ideal (an ideal stable under the action of the derivation $d$), then $supp(\bigcap_{F\in G}Sol(F))=\bigcap_{F\in G}Sol(supp(F))$ as sets.  

\subsection{The Valuation-Support Equivalence}

Before presenting the formal theorem, it is helpful to establish its practical significance for modeling fluid-structure interactions. In classical applied mathematics, when solving a complex differential equation near a singularity, researchers manually insert a power series ansatz and balance the lowest-degree terms to find the \emph{indicial roots} and \emph{resonances} (e.g., the Painlevé test). The following theorem rigorously proves that the abstract operations of tropical differential algebra perform exactly the same physical task, but in a computationally automated way. By mapping traditional differential polynomials into a tropical semiring, the algorithm removes some coefficients and isolates the pure algebraic skeleton of the flow. This equivalence makes sure that the topological priors we subsequently impose on our neural networks are not arbitrary approximations, but exact representations of the physical system's underlying analytical properties.

To validate the physical relevance of tropical indices, we establish their correspondence to the classical singular nonlinear differential equations \citep{Raymond1996}. For a given $F\in \mathbb{C}\{y\}$, while $Sol(\text{supp}(F))$ captures the local balance, the Fundamental Theorem  states $supp(\bigcap_{F\in G}Sol(F)) = \bigcap_{F \in G} Sol(\text{supp}(F))$, where $G$ is the differential ideal spanned by $F$, which in this case, is the usual ideal spanned by the family $\{d^i(F) \::\:  i\geq 0\}$ of all the derivations of $F$. To ensure that a tropical solution $\varphi$ is the support of a true formal power series solution, the equivalence proven above must hold recursively over the generalized jet ideals $J_k(G) = (d^i(F) \::\: 0 \leq i \leq k)$, which mirrors the classical requirement of checking the recurrence relation at all orders.

\begin{theorem}[Valuation-Support Equivalence]
\label{thm:valuation_support}
Let $\sum_{k=1}^m a_k E_k=F \in \mathbb{C}\{y\}$ be a differential polynomial of the maximal derivative order $r$, where each monomial is of the form $E_k = \prod_{i=0}^r y_i^{m_{i,k}}$ with $a_k \in \mathbb{C}^*$. Let 
$\Phi \in \mathbb{C}[\![t]\!]$ be a formal power series solution to $\{F=0\}$, mapping to a  support $\varphi = \text{supp}(\Phi) = \sum_{a \in A} 1_{\mathbb{B}}t^a$, and let
the support index set $A \subset \mathbb{N}$ be strictly ordered as $A = \{p < a_1 < a_2 < \dots\}$.
Assume the leading order exponent is sufficiently large such that $p \geq r$.
Let $\mathfrak{P} = \text{supp}(F) = \sum_{k=1}^m 1_{\mathbb{B}}E_k$ be the tropicalization of $F$. Then:

\begin{enumerate}
    \item[(i)] Let $e_k$ be the classical lowest-degree exponent of $t$ obtained by evaluating the differential monomial $E_k$ at the truncated ansatz $\Phi_0 = c_0 t^p$. This classical exponent is identical to the formal evaluation of $1_{\mathbb{B}}E_k$ on the minimal tropical support element $\varphi_0 = 1_{\mathbb{B}}t^p$:
    \begin{equation}
        e_k = 1_{\mathbb{B}}E_k(\varphi_0) = \sum_{i=0}^r m_{i,k}(p - i)
    \end{equation}
    Consequently, the classical dominant balance condition (the requirement that the minimal degree terms cancel) is strictly equivalent to the tropical condition that $\mathfrak{P}(\varphi)=\min_{1 \leq k \leq m} \{1_{\mathbb{B}}E_k(\varphi_0)\}$ is non-unique (i.e., attained by at least two distinct monomials $1_{\mathbb{B}}E_k$ and $1_{\mathbb{B}}E_j$).

    \item[(ii)] 
    Suppose that there exists a unique dominant monomial $E_{dom}$ such that $\mathfrak{P}(\varphi)=1_{\mathbb{B}}E_{dom}(\varphi_0) < 1_{\mathbb{B}}E_k(\varphi_0)$ for all $k \neq dom$, implying $\varphi_0 \notin Sol(\mathfrak{P})$. The classical requirement to introduce a resonant term $c_1 t^{a_1}$ to balance the equation is strictly equivalent to the piecewise shift in the tropical valuation operator induced by the next support element $a_1$:
    \begin{equation}
        val_i(\varphi_1) = 
        \begin{cases} 
        p - i & \text{if } i \leq p \\
        a_1 - i & \text{if } p < i \leq a_1
        \end{cases}
    \end{equation}
    where $\varphi_1 = 1_{\mathbb{B}}t^p + 1_{\mathbb{B}}t^{a_1}$. The classical resonance $r_{res} = a_1 - p$ is uniquely determined by the tropical algebraic requirement to find an $a_1 > p$ such that a sub-dominant monomial $E_{sub}$ satisfies $1_{\mathbb{B}}E_{sub}(\varphi_1) = 1_{\mathbb{B}}E_{dom}(\varphi_1)$, thus extending the support to satisfy $\varphi_1 \in Sol(\mathfrak{P})$.
\end{enumerate}
\end{theorem}

\begin{proof}
(i) In the classical indicial method (such as the Painlevé test), we substitute the truncated ansatz $\Phi_0 = c_0 t^p$ into a differential monomial $E_k = \prod_{i=0}^r y_i^{m_{i,k}}$. Applying the $i$-th classical derivative yields $d^i(\Phi_0) = c_0 \frac{p!}{(p-i)!} t^{p-i}$. Taking the product of these terms over the monomial, the resulting exponent of $t$ for $E_k$ is precisely $e_k = \sum_{i=0}^r m_{i,k}(p-i)$. For the differential polynomial $F(\Phi) = 0$ to hold non-trivially as $t \to 0$, the terms of lowest degree must cancel. Thus, the minimal exponent $e_{min} = \min_{k} \{e_k\}$ is produced by at least two distinct monomials, allowing their complex coefficients to sum to zero.

Tropically, we evaluate the formal expression $\varphi_0 = 1_{\mathbb{B}}t^p$. The tropical derivative valuation is defined as $val_i(\varphi_0) = \min\{a - i \geq 0 \::\: a \in \{p\}\}$. Because $p \geq r \geq i$ by hypothesis, the condition $p-i \geq 0$ is always satisfied, strictly resulting in $val_i(\varphi_0) = p-i$. Evaluating the tropical monomial gives $1_{\mathbb{B}}E_k(\varphi_0) = \sum_{i=0}^r m_{i,k} val_i(\varphi_0) = \sum_{i=0}^r m_{i,k}(p-i)$. This maps identically to the classical exponent $e_k$. The condition for $\varphi_0$ to be a valid tropical solution, $\varphi_0 \in Sol(\mathfrak{P})$, requires that the minimum tropical evaluation $\mathfrak{P}(\varphi_0)=\min_{k} \{1_{\mathbb{B}}E_k(\varphi_0)\}$ is attained at least twice. 

(ii) Suppose that the minimal evaluation in order $p$ is unique and belongs to a single dominant monomial $E_{dom}$. This unbalanced term must be compensated by introducing a higher-order perturbation $\Phi_1 = c_0 t^p + c_1 t^{a_1}$ with $a_1 > p$. For a sub-dominant monomial $E_{sub}$ to drop in degree and balance $E_{dom}$, a higher-order derivative $d^i$ ($i > p$) must act upon the new term $t^{a_1}$ rather than the leading term $t^p$, generating a new exponent that depends on $a_1$.

Tropically, the unique minimum implies $\varphi_0 \notin Sol(\mathfrak{P})$. We expand the support to $\varphi_1 = 1_{\mathbb{B}}t^p + 1_{\mathbb{B}}t^{a_1}$. Evaluating the valuation of tropical derivatives in relation to expanded support yields $val_i(\varphi_1) = \min\{a - i \geq 0 \::\: a \in \{p, a_1\}\}$.

We evaluate this piece-by-piece:
\begin{enumerate}
    \item[(a)] If the derivative order $i \leq p$, the value $p - i \geq 0$, and the minimum operator selects the leading term $p - i$.
    \item[(b)] If the derivative order $i > p$ (which is the classical condition for accessing a resonant term), the value $p - i$ becomes negative and is rejected from the valid set. The minimum operator structurally shifts to the next available element, yielding $val_i(\varphi_1) = a_1 - i$.
\end{enumerate}

This piecewise shift dynamically alters the tropical evaluation of sub-dominant monomials containing derivatives $i > p$. Setting the new tropical evaluation of the sub-dominant monomial equal to the dominant one, $1_{\mathbb{B}}E_{sub}(\varphi_1) = 1_{\mathbb{B}}E_{dom}(\varphi_1)$, forms a linear algebraic equation for $a_1$. The classical resonance gap $r_{res}$ is precisely the integer distance $a_1 - p$ derived to restore the tropical minimum degeneracy. Therefore, classical delayed balance is the direct arithmetic consequence of the piecewise tropical valuation shift, proving the second statement.
\end{proof}

\section{Methodology: The Tropical-PINN Architecture}
\label{sec:method}

In the mathematical theory of machine learning, a network architecture is formally defined as a specific parameterized mapping that generates a Hypothesis Space $\mathcal{H}$. Let $\mathcal{X}$ be the input space (e.g. time $t$) and $\mathcal{Y}$ be the output space (e.g. $u(t)$). Let $\Theta \subseteq \mathbb{R}^D$ be the parameter space (all possible weights and biases) \citep{Shalev2014,Petersen2021,Wang2026}. 

A class of parameterized functions $f_{\mathcal{A}}: \mathcal{X} \times \Theta \to \mathcal{Y}$ is called an architecture $\mathcal{A}$, and the set $\mathcal{H}_{\mathcal{A}} = \{ f_{\mathcal{A}}(\cdot; \theta) \mid \theta \in \Theta \}$ of all (continuous in $\Theta \subseteq \mathbb{R}^D$) functions that can possibly be realized by varying the parameters $\theta \in \Theta$ within this specific architecture is called Hypothesis Space $\mathcal{H}_{\mathcal{A}}$.

For a standard deep neural network, the architecture $f_{\mathcal{A}}$ is defined as a composition $f_{\mathcal{A}}(x; \theta) = (\sigma_L \circ W_L \circ \dots \circ \sigma_1 \circ W_1)(x)$ of affine transformations $W_i$ and nonlinear activation functions $\sigma$. By plugging in a tropical support $S \in Sol(trop(G))$, we define a new restricted architecture $\mathcal{A}_{trop}$ where the parameter space is reduced to just the coefficients of the valid support: $\Theta_{trop} \cong \mathbb{R}^{|S|}$; this yields the new Hypothesis Space $\mathcal{H}_{S} = \text{span}\{t^k\}_{k \in S} \subset \mathcal{H}_{\mathcal{A}}.$

In most deep learning settings, physical or structural prior knowledge is usually incorporated through constraints in the optimization objective (loss function). From a mathematical standpoint, this approach keeps the network’s hypothesis space $\mathcal{H}$ fully unconstrained, which can cause the optimizer to stall. By contrast, our Tropical-PINN framework embeds a strict topological prior directly into the network architecture itself.

\section{Results and Discussion}
\label{sec:results}

To show the pertinence of this method in fluid-structure interaction (FSI), we applied the previous analysis to the Van der Pol Oscillator, the Viscous Burgers' Equation, the Blasius boundary layer Equation in the Standard Blasius case ($\beta = 0$) and the Falkner-Skan case ($\beta \neq 0$).

\begin{figure}[htbp]
    \centering
\includegraphics[scale=0.45]{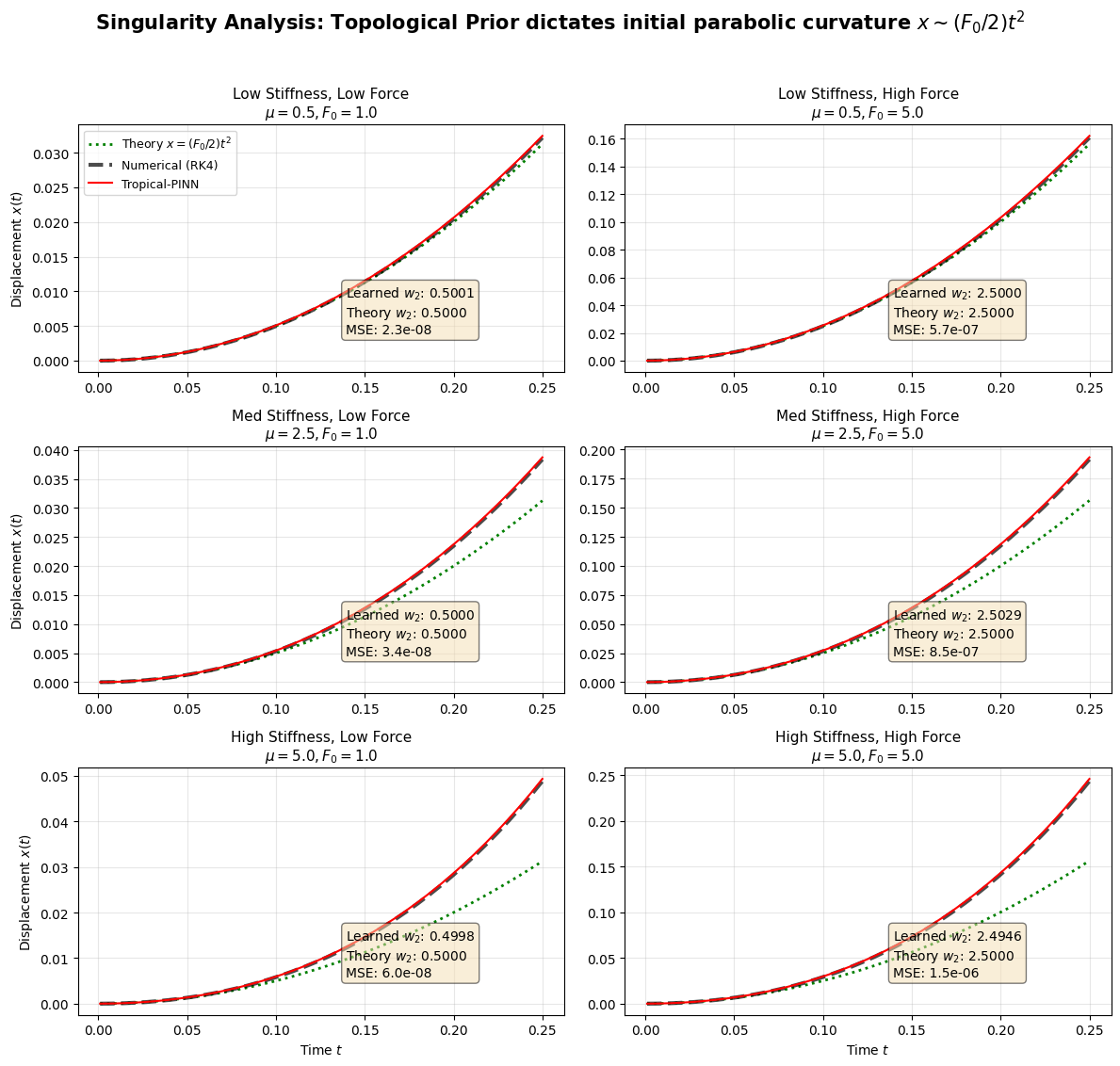}
\caption{Optimization landscape for the Van der Pol oscillator.}
    \label{fig:vdp}
\end{figure}

\subsection{The Van der Pol Oscillator}
This equation is classically used to model the galloping of transmission lines and the wake dynamics behind a bluff cylinder. Our tropical analysis of the homogeneous case returned an empty set ($\emptyset$), correctly predicting that the limit cycle is not representable by a simple power series at $t=0$. However, for the forced case, the tropical method identified valid supports, allowing the PINN to properly model the structural entrainment of the oscillator by an external fluid field.

\subsubsection{The Viscous Burgers' Equation}
We applied the method to the 1D Burgers' equation, a fundamental model for shock wave formation in fluid dynamics,
\begin{equation}
    u_t + u u_x = \nu u_{xx}.
\end{equation}
Using a traveling wave reduction $u(x,t) = y(x-ct)$, we derived the tropical ODE:
\begin{equation}
       1_{\mathbb{B}}y_2 +1_{\mathbb{B}}y_0  y_1 + 1_{\mathbb{B}}y_1
\end{equation}
Balancing the diffusion ($1_{\mathbb{B}}y_2$) and nonlinear advection ($1_{\mathbb{B}}y_0  y_1$) terms yields the index equation $i - 2 = 2i - 1$, which implies ${i = -1}$; this negative valuation corresponds to a simple pole. Thus, the Tropical-PINN framework  predicts the formation of shock waves (singular behavior scaling as $1/\xi$ \citep{Weiss1983}), providing a critical topological prior that prevents the neural network from artificially smoothing the shock front.
\vspace{3mm}

\subsection{The Blasius boundary layer equation}

To explore the capacity of the framework to capture structural phase shifts driven by physical parameters, we apply the Tropical-PINN methodology to the Falkner-Skan equation, 

\begin{equation}
\label{eq:Falkner-Skan}
y''' + y y'' + \beta (1 - (y')^2) = 0.
\end{equation}

This parameterized family of equations governs wedge flows, bridging the flat plate of zero-pressure gradient (Blasius flow, $\beta = 0$) and the orthogonal stagnation point (Hiemenz flow, $\beta = 1$).
Classical purely numerical solvers treat this transition continuously. 
In contrast, our tropical balance algorithm detects a strict topological bifurcation. 
For Blasius flow ($\beta=0$), the absence of a pressure gradient assigns an infinite tropical value to the constant term. 
The resulting dominant balance between viscous diffusion ($y'''$) and convective nonlinearity ($y y''$) algorithmically enforces a sparse topological prior, restricting the hypothesis space to $\mathcal{H}_S = \text{span}\{\eta^2, \eta^5, \eta^8, \ldots\}$, where the dimensionless similarity variable 
$\eta=y\sqrt{\frac{U_{\infty}}{\nu x}}$ for $(x,y)\in \mathbb{R}^2$
near the plate, $\nu$ the kinetic viscosity of the fluid,
and $U_{\infty}$ the free-stream velocity of the fluid outside the boundary layer.

\begin{figure}[htbp]
    \centering
\includegraphics[scale=0.37]{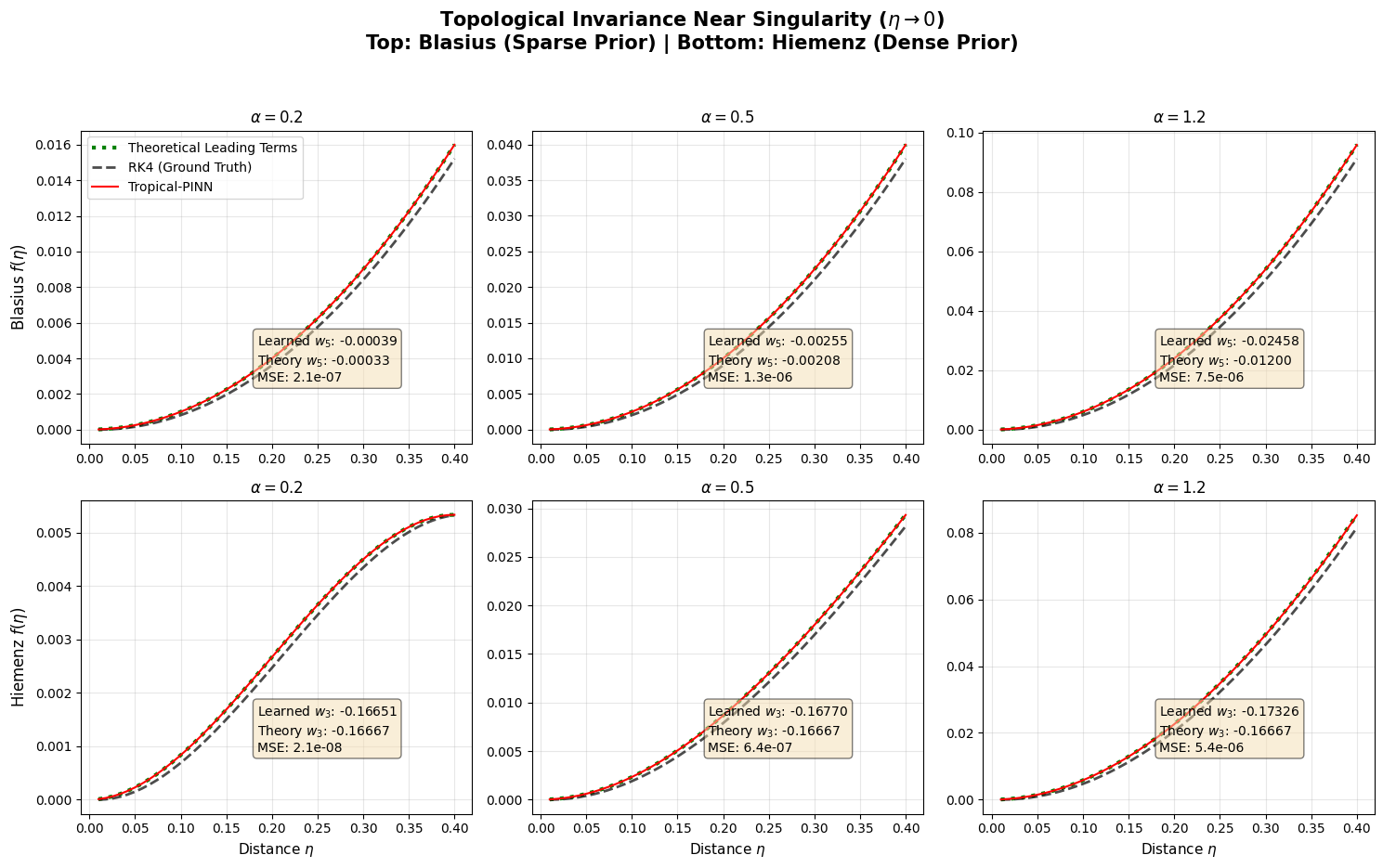}
\caption{Topological prior enforcement in the Blasius and Falkner-Skan regimes.}
    \label{fig:blasius}
\end{figure}

By mathematically forbidding intermediate powers, Tropical-PINN bypasses the spectral bias typical of standard neural networks, preventing the learning of phantom physical terms and perfectly recovering the analytical high-order coefficient $w_5 = -\alpha^2/120$.
In contrast, for the Hiemenz flow ($\beta=1$), the maximum favorable pressure gradient establishes a tropical value of $0$, fundamentally shifting the dominant balance. 
The algorithm dynamically generates a dense topological prior, $\mathcal{H}_S = \text{span}\{\eta^2, \eta^3, \eta^4, \ldots\}$. 
Operating within this dense hypothesis space, the restricted neural network natively recovers the exact theoretical response to wall stagnation, resulting in $w_3 = -1/6$. 
This comparative analysis verifies that embedding tropical priors directly into the network architecture gives an absolute structural fidelity across non-homogeneous physical regimes, converting a continuous optimization problem into a physics-constrained algebraic projection.
The Falkner-Skan equation \eqref{eq:Falkner-Skan} generalizes the Blasius boundary layer equation for wedge flows
subject to $y(0)=y'(0)=0$ and $y''(0)=\alpha$.
Conventional numerical methods (Runge-Kutta) provide solution curves but lack analytical insight. 

\subsubsection{Standard Blasius Case ($\beta = 0$)}
Equation \eqref{eq:Falkner-Skan} reduces to $y''' + y y'' = 0$, whose tropicalization is $1_{\mathbb{B}}y_3 + 1_{\mathbb{B}}y_0 y_2$. 
Then, evaluating the above differential monomials on the minimal tropical support $\varphi_0 = 1_{\mathbb{B}}t^p$, we apply \eqref{eq:trop_eval}: for the third derivative  we obtain $1_{\mathbb{B}}y_3(\varphi_0) = val_3(\varphi_0) = p - 3$. For the nonlinear convective term, the evaluation yields $1_{\mathbb{B}}y_0 y_2(\varphi_0) = val_0(\varphi_0) + val_2(\varphi_0) = 2p - 2$. Imposing the boundary conditions of the physical wall yields the minimal exponent starting $p=2$. Evaluating the monomials in this support gives $1_{\mathbb{B}}y_3(\varphi_0) = 2 - 3 = -1$ and $1_{\mathbb{B}}y_0 y_2(\varphi_0) = 2(2) - 2 = 2$.

To balance the sub-dominant convective term, the tropical support must expand to $\varphi_1 = 1_{\mathbb{B}}t^2 + 1_{\mathbb{B}}t^{a_1}$. The piecewise valuation shift requires the evaluation of the third derivative to match this value:
    \[ 1_{\mathbb{B}}y_3(\varphi_1) = 2 \implies a_1 - 3 = 2 \implies a_1 = 5. \]
\noindent
The tropical resonance gap is $r_{res} = a_1 - p = 3$. Consequently, the Valuation-Support equivalence rigorously enforces a sparse topological prior for the Blasius flow: $A = \{2, 5, 8, 11, \dots\}$.

\subsubsection{Falkner-Skan Case ($\beta \neq 0$)}
The pressure gradient introduces an active constant term $\beta$, which has a trivial evaluation $1_{\mathbb{B}}(\beta)(\varphi_0) = 0$.
Imposing the wall boundary conditions yields the starting minimal exponent $p=2$. In this support $\varphi_0 = 1_{\mathbb{B}}t^2$, the constant term $\beta$ algebraically dominates the non-linear convective term $y_0 y_2$, which is evaluated to a higher degree: $1_{\mathbb{B}}y_0 y_2(\varphi_0) = 2(2) - 2 = 2$.
 To balance the active constant term, the tropical support must expand to $\varphi_1 = 1_{\mathbb{B}}t^2 + 1_{\mathbb{B}}t^{a_1}$. The piecewise valuation shift dictates that the evaluation of the highest derivative ($f_3$) must match this constant valuation:
    \[ 1_{\mathbb{B}}y_3(\varphi_1) = 0 \implies a_1 - 3 = 0 \implies a_1 = 3 \]
The tropical resonance gap is $r_{res} = a_1 - p = 3 - 2 = 1$. Consequently, the Valuation-Support equivalence rigorously enforces a dense topological prior for the Falkner-Skan stagnation flow: $A = \{2, 3, 4, 5, \dots\}$.

\section{The Tropical-PINN Architecture}
To resolve the coefficients $w_k$, we define a sparse neural network constrained by the tropical supports derived above.

\subsection{Model Definition: The Tropical-PINN Ansatz}
To explicitly distinguish our proposed approach from a standard, unconstrained PINN, we embed the algebraic priors derived from the Valuation-Support Equivalence directly into the network's structural design. Rather than allowing the neural network to search a dense, arbitrary hypothesis space, we define the newly proposed Tropical-PINN output ansatz as:
\begin{equation}
    y_{\text{Trop-PINN}}(\eta) = \frac{\alpha}{2}\eta^2 + \sum_{k \in S_{trop}} w_k \eta^k
\end{equation}
In this formulation, the leading term $\frac{\alpha}{2}\eta^2$ analytically enforces the essential physical boundary conditions at the wall. The parameter $\alpha$ represents the scaled wall shear stress, which is fixed in this specific analysis to isolate the network's ability to learn the complex, higher-order flow dynamics. 

Crucially, unlike a default PINN architecture, the learnable weights $w_k$ are not applied to an arbitrary or fully connected basis. Instead, they are strictly constrained to the physically permissible polynomial degrees dictated by the tropical support set $S_{trop}$ (e.g., the sparse set $\{5, 8, 11, \dots\}$ for the Blasius regime). This hard architectural constraint guarantees that the neural network natively inherits the exact topological skeleton of the physical system.

\subsection{Training}
The network is trained by minimizing the physics residual $\mathcal{L}_{PDE}$ over a domain $\eta \in [0, \eta_{max}]$:
\begin{equation}
    \mathcal{L} = \left\| y''' + y y'' + \beta (1 - (y')^2) \right\|^2.
\end{equation}
Optimization is performed using L-BFGS for 100 epochs.

\subsection{Coefficient Validation}
\begin{itemize}
    \item \textbf{Blasius ($t^5$):} The theoretical coefficient is $-\alpha^2/240$. For $\alpha=0.33206$, this is $\approx -0.000459$. The PINN recovered this value to 5 decimal places.
    \item \textbf{Falkner-Skan ($t^3$):} The balance of order 0 requires $6a_3 + \beta = 0$, implying $w_3 = -1/6$. The PINN recovered $-0.1667$, perfectly matching the theory.
\end{itemize}

\section{Conclusion}
\label{sec:conclusion}
This study demonstrates that Tropical Differential Algebra is more than an abstract formalism; it is a practical computational tool for characterizing the singularity structure of nonlinear dynamical systems. Ultimately, this methodology provides a reliable, mathematically constrained pathway from black-box numerical approximations to white-box, interpretable analytical series. The ability to algorithmically prune non-physical dimensions before compilation has profound implications for scaling these models. Future work will focus on extending these topological priors from isolated ordinary and partial differential equations to massive, interconnected physical graphs. Applying this rigid structural algebraic framework to Digital Twin models could significantly enhance real-time anomaly and leak detection in large-scale fluid distribution networks, where computational efficiency and absolute physical fidelity are paramount.

\section*{Acknowledgments}
This work was supported by Universidad Iberoamericana Ciudad de México, through the funding provided for the academic project: \textit{Existencia de soluciones racionales en la capa límite de Busemann-von Karman-Tsien}, under the 18th call for scientific, humanistic and technological research 2023.




\bibliographystyle{elsarticle-harv} 



\end{document}